 \documentclass[brochure,english,12pt]{smfbourbaki}

\usepackage[T1]{fontenc}
\usepackage{lmodern,amssymb,amsmath, bm}

\usepackage{babel}
\usepackage{scrtime}
\usepackage{mathrsfs}
\usepackage[shortlabels]{enumitem}
\usepackage[all]{xy}

\usepackage[utf8]{inputenc}

\usepackage[colorlinks=true, linkcolor=blue, citecolor=red,
urlcolor=blue]{hyperref}

\date{Novembre 2024}
\bbkannee{77\textsuperscript{e} année, 2024--2025}  
\bbknumero{1228}                                      

\title{Smoothing low-dimensional algebraic cycles}
\subtitle{after Koll\'ar and Voisin}

\author{Olivier Benoist}
\address{DMA, \'ENS et CNRS, \\
45 rue d'Ulm, \\
75005 Paris, FRANCE}
\email{olivier.benoist@ens.fr}

\numberwithin{equation}{section}
\theoremstyle{plain}
\newtheorem{thm}[equation]{Theorem}

\newtheorem{lem}[equation]{Lemma}
\newtheorem{que}[equation]{Question}
\newtheorem{pro}[equation]{Proposition}
\theoremstyle{definition}
\newtheorem{de}[equation]{Definition}
\theoremstyle{remark}

\newtheorem{Step}{Step}

\newcommand{\arxiv}[1]{\href{https://arxiv.org/abs/#1}{arXiv:#1}}

\newcommand{\degr}{\mathrm{deg}}
\newcommand{\rk}{\mathrm{rk}}
\newcommand{\Pic}{\mathrm{Pic}}

\newcommand{\Hdg}{\mathrm{Hdg}}
\newcommand{\CH}{\mathrm{CH}}
\newcommand{\KV}{\mathrm{KV}}
\newcommand{\bP}{{\mathbb P}}
\newcommand{\Z}{{\mathbb Z}}
\newcommand{\Q}{{\mathbb Q}}
\newcommand{\R}{{\mathbb R}}
\newcommand{\C}{{\mathbb C}}
\newcommand{\ci}{\mathcal{C}^{\infty}}
\newcommand{\tZ}{{\widetilde Z}}
\newcommand{\sL}{{\mathcal L}}

\newcommand{\sF}{{\mathcal F}}
\newcommand{\sE}{{\mathcal E}}

\newcommand{\sO}{{\mathcal O}}

\begin{document}

\maketitle

\section*{Introduction}

\subsection{The smoothing problem for algebraic cycles}

Let $k$ be a field of characteristic $0$ (which the reader may choose to be the field~$\C$ of complex numbers).
Let $X$ be a smooth projective algebraic variety of dimension $n$ over~$k$.
A central theme in algebraic geometry is the study of algebraic cycles on $X$, that is of the 
collection of all the algebraic subvarieties of $X$. We will always denote by~$d$ the dimension of the subvarieties we consider.

The main players of this line of research are the Chow groups $\CH_d(X)=\CH^{n-d}(X)$ 
of~$X$. Their elements are \textit{cycles}:
linear combinations with integral coefficients of (closed, integral) $d$\nobreakdash-dimensional algebraic subvarieties of $X$,  considered up to \textit{rational equivalence}. Here, two cycles are said to be rationally equivalent if both belong to the same algebraic family of cycles on $X$ parametrized by the projective line~$\bP^1_{k}$.
We refer to~\cite{Fulton} for precise definitions,  as well as for a study of functorial properties of Chow groups.

One should think of these groups as algebro-geometric analogues of singular homology groups, where both the generators and the relations are constrained to have an algebraic (as opposed to topological) origin.  No conditions are imposed on the algebraic subvarieties of $X$ that generate its Chow groups; in particular, they may carry arbitrary singularities. The next question, first asked
by Borel and Haefliger
 in \mbox{\cite[\S 5.17]{BH}}, is therefore of interest.

\begin{que}
\label{qBH}
Let $X$ be a smooth projective algebraic variety over $k$.  Are the Chow groups $\CH_*(X)$ of $X$ generated by classes of smooth subvarieties of $X$?
\end{que}

In other words, is it possible to smooth algebraic cycles up to rational equivalence? To be precise, the original question of Borel and Haefliger was slightly weaker.  They considered,  over the field $k=\C$, the coarser \textit{homological equivalence} relation, where two algebraic cycles are identified if they have the same image by the cycle class map~$\CH_d(X)\to H_{2d}(X,\Z)$. 

One of the goals of this introduction is to review the state of the art concerning Question \ref{qBH} (see~\S\ref{parpos} for positive results and \S\ref{parnon} for negative results).

\subsection{Smoothing algebraic cycles in the Whitney range}
\label{parWhitney}

Let us state right away the main theorem presented in this survey, which has been recently proved by Koll\'ar and Voisin  \cite[Theorem 1.2]{KV}.

\begin{thm}
\label{thKV}
Let $X$ be a smooth projective algebraic variety of dimension~$n$ over~$k$. If $d<\frac{n}{2}$, then $\CH_d(X)$ is generated by classes of smooth subvarieties of~$X$.
\end{thm}

At the time when Borel and Haefliger asked Question \ref{qBH},  resolution of singularities was not available.  A positive answer to this question could therefore have been used as a substitute to resolution of singularities in the study of algebraic cycles.  This original motivation has nowadays disappeared.  On the contrary,  the proof of Theorem~\ref{thKV} does use Hironaka's theorem~\cite{Hironaka} on resolution of singularities (exactly once, in the proof of Theorem \ref{th2'} given at the end of \S\ref{flatci}).  

It is verified in \cite[Theorem 39]{Kollar} that the proof of Theorem \ref{thKV} can be extended to base fields $k$ of characteristic $p\geq n-d$,  at least if $k$ is assumed to be infinite and perfect. 
In this argument, resolution of singularities is replaced with Gabber's improvement \mbox{\cite[Theorem 2.1]{Gabber}} of de Jong's alteration theorem.

Let us explain the significance and the importance of the restrictive hypothesis~$d<\frac{n}{2}$ in the statement of Theorem~\ref{thKV}.
It plays the  exact same role as in Whitney's weak embedding theorem \cite{Whitney} in differential topology (according to which any compact~$\ci$ manifold of dimension $d$ embeds in $\R^n$ if $d<\frac{n}{2}$).
The heuristic in our algebraic situation is that a morphism $f:Y\to X$ of smooth projective varieties over $k$ which is sufficiently generic is an embedding if $\dim(Y)<\frac{n}{2}$.  
Its image $f(Y)$ is then a smooth subvariety of~$X$.  This suggests that,  under the hypothesis $d<\frac{n}{2}$, the variety $X$ should contain many smooth subvarieties constructed by general projection arguments,  therefore increasing the likelihood that $\CH_d(X)$ is generated by classes of smooth subvarieties of $X$. Of course, the difficulty is to enforce this genericity condition on $f$ by algebraic means.

A Whitney-type hypothesis had already appeared in two earlier works on Question~\ref{qBH}. On the one hand, Hironaka had given a proof of Theorem \ref{thKV} under the additional assumption that $d\leq 3$ (see \cite[Theorem p.~50]{Hironaka}). On the other hand, I had constructed counterexamples to Question \ref{qBH} on the boundary $d=\frac{n}{2}$ of the Whitney range, for infinitely many values of $d$ (see \cite[Theorem 0.3]{Ben}), thereby showing that the hypothesis $d<\frac{n}{2}$ in Theorem \ref{thKV} is optimal (for these values of~$d$).  These works will be discussed in more details in \S\ref{lowd} and \S\ref{count} respectively.

\subsection{Flat pushforwards of complete intersections}
\label{flatci}

The approach of Koll\'ar and Voisin relies chiefly on the following definition. Recall that a morphism $f:Y\to X$ of connected smooth projective varieties over $k$ is \textit{flat} if and only if it is equidimensional, i.e.\ if all its fibers have dimension ${\dim(Y)-\dim(X)}$.

\begin{de}
\label{defKV}
Let $X$ be a smooth projective variety over $k$. Define $\CH_d(X)_{\KV}$ to be the subgroup of $\CH_d(X)$ generated by cycles of the form $f_*(\lambda_1\dots\lambda_c)$ for some flat morphism ${f:Y\to X}$ of smooth projective varieties over $k$ and some codimension~$1$ classes $\lambda_1,\dots,\lambda_c\in\CH^1(Y)$, where $c+d=\dim(Y)$.
\end{de}

In short, the subgroup $\CH_d(X)_{\KV}\subset\CH_d(X)$ is generated by those cycles that may be written as flat pushforwards of intersections of divisor classes. Theorem \ref{thKV} results from the combination of the following two theorems.

\begin{thm}[{\cite[Proposition 1.5]{KV}}]
\label{th1}
Let $X$ be a smooth projective variety of dimension $n$ over $k$. If~${d<\frac{n}{2}}$, then any element of $\CH_d(X)_{\KV}$ may be written as a linear combination with integral coefficients of classes of smooth subvarieties of $X$.
\end{thm}

\begin{thm}[{\cite[Theorem 1.6]{KV}}]
\label{th2}
If $X$ is a smooth projective variety over $k$, then $${\CH_*(X)_{\KV}=\CH_*(X)}.$$
\end{thm}

Theorem \ref{th1} is the part of the proof of Theorem \ref{thKV} where the smoothing actually takes place.  It is also the part in which the Whitney-type hypothesis $d<\frac{n}{2}$ is used.  It is directly inspired by Hironaka's work  \cite{Hironaka} on the topic. We discuss it more in~\S\ref{lowd}, and present its proof in Section \ref{s1}.

In contrast, Theorem \ref{th2} holds with no restriction on the dimension of the cycles. It constitutes a fundamentally new structural result on the Chow groups of arbitrary smooth projective varieties.   Its formulation and its proof are the main achievements of the article \cite{KV}.
It is not known whether Theorem \ref{th2} would still hold if one required the morphisms~$f$ in Definition \ref{defKV} to be smooth instead of only flat (see \cite[Question~1.11]{KV}).
The flexibility gained by allowing flat morphisms that are possibly not smooth is used exactly once in the proof of Theorem \ref{th2} (in the proof of Proposition \ref{hyp}).

As noted in \cite{Kollar}, the proof of Theorem \ref{th2} given in \cite{KV} yields a slightly stronger result, valid at the level of subvarieties (as opposed to Chow groups).  
To state it, we rely on the next geometric definition. Recall that a subvariety of codimension $c$ in a smooth projective variety $X$ over $k$ is said to be a \textit{complete intersection} if it can be written as the intersection of $c$ hypersurfaces in $X$.

\begin{de}
\label{defKV2}
Let $X$ be a smooth projective variety over $k$.  An integral subvariety~$Z\subset X$ is said to be a smooth complete intersection image (or \textit{sci-image} for short) if there exist a flat morphism $f:Y\to X$ of smooth projective varieties over $k$ and a smooth complete intersection $V\subset Y$ such that $f(V)=Z$ and $f|_V:V\to Z$ is birational.
\end{de}

\begin{thm}[{\cite[Theorem 2]{Kollar}}]
\label{th2'}
Let $X$ be a smooth projective variety over $k$. Any integral subvariety $Z\subset X$ is an sci-image.
\end{thm}

Theorem \ref{th2} is an immediate consequence of Theorem \ref{th2'}. In turn, the proof of Theorem \ref{th2'} has two steps. In the first step, one studies functorial properties of sci-images. The culmination of this analysis is the next proposition.   In its statement, one makes use of the following terminology: a subvariety of codimension $c$ in a smooth projective variety $X$ over $k$ is said to be a \textit{complete bundle section} (or \textit{cbs} for short) if it can be written as the zero locus of a section of a vector bundle of rank $c$ on $X$.

\begin{pro}[\mbox{\cite[Proposition \hspace{-.1em}3.11]{KV}, \cite[Lemma \hspace{-.1em}16]{Kollar}}]
\label{pro1}
Let $\pi:X'\to X$ be the blow-up of a smooth cbs in a smooth projective variety over $k$.  Let $Z'\subset X'$ be an sci-image such that $\pi|_{Z'}:Z'\to Z:=\pi(Z')$ is birational. Then $Z\subset X$ is an sci-image.
\end{pro}

To conclude the proof of Theorem \ref{th2'}, one would like to show that any integral subvariety of $X$ can be turned into a smooth complete intersection after repeatedly blowing up smooth cbs.  This is not known to hold (see e.g.\ \cite[Question~4.1]{KV}) and the second step is a weaker statement, which is sufficient for our purposes.

\begin{pro}
[\mbox{\cite[Theorem 1.9]{KV}}]
\label{pro2}
Let $X$ be a smooth projective variety of dimension $n$ over $k$. Let $Z\subset X$ be a smooth subvariety of dimension $<\frac{n}{4}$. Then there exist a composition $\pi^+:X^+\to X$ of blow-ups of smooth cbs and a smooth complete intersection $V\subset X^+$ such that $\pi^+(V)=Z$ and $\pi^+|_{V}:V\to Z$ is birational.
\end{pro}

Proposition \ref{pro2} is the technical heart of \cite{KV}.
Propositions \ref{pro1} and \ref{pro2} will be proved in Sections \ref{s2} and \ref{s3} respectively. 
Together, they imply Theorem \ref{th2'}, as we now explain.

\vspace{.5em}

\begin{proof}[Proof of Theorem \ref{th2'}]
Let $\tZ\to Z$ be a resolution of singularities (see \cite{Hironaka}).  Embed~$\tZ$ in $\bP^N_k$ for some $N\geq 0$. 
View $\tZ$ as a subvariety of $Y:=X\times\bP^N_k$ using the natural diagonal embedding. After possibly increasing $N$, one can apply Proposition \ref{pro2} to the subvariety~$\widetilde{Z}$ of $Y$.  In this way, we obtain a composition $\pi^+: Y^+\to Y$ of blow-ups of smooth cbs and a smooth complete intersection $V\subset Y^+$ such that $\pi^+(V)=\widetilde{Z}$ and~$\pi^+|_{V}:V\to \widetilde{Z}$ is birational.  Using Proposition \ref{pro1}, we deduce that the subvariety~$\tZ\subset Y$ is an sci-image.  It follows that $Z\subset X$ is also an sci-image, by flatness of the first projection morphism $Y=X\times\bP^N_k\to X$.
\end{proof}

\subsection{Smoothing techniques}
\label{parpos}

We now focus on the various techniques that have been used to smooth algebraic cycles up to rational equivalence,  both in the early positive results about Question~\ref{qBH} and in the recent work of Koll\'ar and Voisin. These techniques fall into two categories,  depending on whether the algebraic cycles under consideration have small dimension, or small codimension. In the first case, they are best thought of homologically and presented as pushforwards (see~\S\ref{lowd}).  In the second case they are best thought of cohomologically and presented as pullbacks (see~\S\ref{lowc}).
Combining the results presented in~\S\ref{lowd} and \S\ref{lowc} shows that Question \ref{qBH} has a positive answer when $n\leq 5$.

\subsubsection{Cycles of small dimension}
\label{lowd}

The first progress on Question \ref{qBH} was due to Hironaka \cite[Theorem p.~50]{Hironaka} who answered it positively when $d<\frac{n}{2}$ (the Whitney-type condition discussed in \S\ref{parWhitney}) and $d\leq 3$. Let us sketch his proof.

Let $Z\subset X$ be an integral subvariety of dimension $d$, and let $\tZ\to Z$ be a resolution of singularities.  
Choosing an embedding of $\tZ$ in $\bP^N_k$ for some $N\geq 0$ allows one to view~$\tZ$ as a subvariety of $Y:=X\times\bP^N_k$. (That Hironaka's argument inspired the proof of Theorem~\ref{th2'} given at the end of \S\ref{flatci} should be obvious.) At that point,  the idea is to find a cycle~$\tZ'$ in $Y$, which is rationally equivalent to $\tZ$, and whose components are smooth and in general position.  The image $Z'$ of $\tZ'$ in $X$ is then rationally equivalent to $Z$, and has smooth components precisely because of the Whitney-type hypothesis.

To construct $\tZ'$ from $\tZ$, one needs some kind of moving lemma. To this effect, Hironaka devises a \textit{moving by linkage} technique.  Let $\sL$ be a sufficiently ample line bundle on $Y$, and let $c$ be the codimension of $\tZ$ in~$Y$.  Let $D_1,\dots,D_c$ be general elements of the  linear system $|\sL|$, and let $E_1,\dots,E_c$ be general elements of $|\sL|$ containing $\tZ$. One can write $E_1\cap\dots\cap E_c=\tZ\cup W$ for some subvariety $W\subset Y$ of codimension $c$.  One says that the subvarieties $\tZ$ and $W$ of $Y$ are \textit{linked}.  Our choices ensure that the cycle~$\tZ':=(D_1\cap\dots \cap D_c)-W$ is rationally equivalent to $\tZ$.  The subvariety $W$ is in general singular in codimension~$4$, so it is smooth when $d\leq 3$.  To complete the proof, repeat the linkage procedure a few times to enforce the general position hypothesis.

For cycles of dimension $\geq 4$, Hironaka's method fails because of the singularities that linkage inevitably creates.  However, combined with additional arguments to control these singularities, these ideas may still be useful  (see \cite[Theorems 0.4 and 0.6]{Ben} for applications to real algebraic cycles).

The perspective of Koll\'ar and Voisin is very different.  They do not attempt to control the singularities that appear. Neither do they try to develop another moving technique applicable to general cycles. Instead,  they prove the structural result that all cycles come by flat pushforward from complete intersections (Theorem \ref{th2}) and they use that complete intersections can be moved around very easily.  A small price to pay is that they need to apply the Whitney-type general projection argument to a morphism~$f:Y\to X$ that is only assumed to be flat, and not smooth as is the first projection~$X\times\bP^N_k\to X$ in Hironaka's proof.
This is the content of Theorem \ref{th1}.

\subsubsection{Cycles of small codimension}
\label{lowc}

The case $d=n$ of Question \ref{qBH} is trivial (as the $n$-dimensional subvarieties of $X$ are its connected components, which are smooth).

It is also true that Question \ref{qBH} has a positive answer for $d=n-1$.  Indeed, if $D\subset X$ is an irreducible divisor, one can write~$\sO_X(D)=\sL_1\otimes\sL_2^{-1}$, where the $\sL_i$ are very ample line bundles on $X$.  General divisors~$D_i$ in the linear systems $|\sL_i|$ are then smooth by the Bertini theorem.  To conclude, write~$[D]=[D_1]-[D_2]$ in $\CH_{n-1}(X)=\Pic(X)$.

The article \cite{Kleiman} of Kleiman can be thought of as extending this argument to higher-codimensional cycles. This led him, in particular, to give positive answers to Question~\ref{qBH} when $n=4$ and $d=2$, and when $n=5$ and $d=3$.

For the sake of completeness,  let us briefly explain Kleiman's argument. The Grothendieck--Riemann--Roch theorem without denominators
shows that 
\begin{equation}
\label{GRR}
[Z]=(-1)^{c-1}(c-1)!\,c_c(\sO_Z)
\end{equation}
in $\CH^c(X)$, for any subvariety $Z\subset X$ of codimension $c$ (see e.g.\ \cite{Jouanolou}). 
Applying~(\ref{GRR}) with~${c=2}$, and using a resolution of the coherent sheaf $\sO_Z$ by locally free sheaves,  one can see that $\CH^2(X)$ is generated (integrally) by second Chern classes of vector bundles.  
Tensoring these vector bundles with high enough powers of a fixed ample line bundle, one may even assume that these vector bundles are globally generated.

In other words, the Chow group $\CH^2(X)$ is generated by classes of the form $f^*c_2(\sE)$, where $f:X\to G$ is a morphism to a Grassmannian, and $\sE$ is the tautological vector bundle on $G$.  The class $c_2(\sE)\in\CH^2(G)$ is represented by a Schubert cell $S\subset G$, whose singularities are in codimension $4$.  Assuming that $f$ is in general position with respect to $S$, which can be ensured by postcomposing $f$ with a generic automorphism of~$G$, the class $f^*c_2(\sE)$ is represented by the subvariety $f^{-1}(S)$ of $X$, which is smooth if $d\leq 3$.

We also note, following \cite[Lemma 2.5]{KV}, that this circle of ideas yields an easy proof of Theorem~\ref{th2} with rational coefficients. Indeed, if $\alpha\in\CH^c(X)$, then (\ref{GRR}) shows that~$(c-1)!\,\alpha$ belongs to the subring of $\CH^*(X)$ generated by Chern classes of vector bundles, and hence to the subring of $\CH^*(X)$ generated by Segre classes of vector bundles (since the total Chern class $c(\sE)$ and the total Segre class $s(\sE)$ of a vector bundle $\sE$ on $X$ are related by the identity $c(\sE)\cdot s(\sE)=1$). It is however obvious, from their definition given in \cite[\S 3.1]{Fulton}, that (products of) Segre classes belong to~$\CH^*(X)_{\KV}$.

\subsection{Nonsmoothable algebraic cycles}
\label{parnon}

At last, we discuss the existing counterexamples to Question \ref{qBH}.

\subsubsection{A topological variant}

 Even before the work of Borel and Haefliger,  an analogue of Question \ref{qBH} in differential topology was considered by Thom in his influential article~\cite{Thom}.

\begin{que}
\label{qThom}
Let $M$ be an oriented compact $\ci$ manifold.  Are the homology groups 
$H_*(M,\Z)$ generated by fundamental classes of oriented $\ci$ submanifolds of $M$?
\end{que}

One of Thom's results is that a class in $H_d(M,\Z)$ can be realized as the fundamental class of an oriented $\ci$ submanifold of dimension $d$ of $M$ if and only if its Poincar\'e dual cohomology class is induced by pullback from some universal cohomology class: the Thom class of the universal oriented real vector bundle of rank~${\dim(M)-d}$ (see~\mbox{\cite[Th\'eor\`eme II.5]{Thom}}). 
The vanishing of certain cohomological operations applied to this Thom class then give rise to restrictions on the fundamental classes of $d$\nobreakdash-dimensional oriented compact~$\ci$ submanifolds of~$M$. Thom used these restrictions in \cite[Th\'eor\`eme~III.9]{Thom} to settle Question \ref{qThom} in the negative in general.

Let us point out that one can use Thom's criterion to verify that Question~\ref{qThom} has a positive answer in the Whitney range (that is, for classes in $H_d(M,\Z)$ with $d<\frac{\dim(M)}{2})$.
The counterpart of Theorem \ref{thKV} in differential topology is therefore true.

\subsubsection{Counterexamples to Question \ref{qBH}}
\label{count}

  By applying Thom's ideas to complex algebraic varieties, Hartshorne, Rees and Thomas \cite[Theorem 1]{HRT} found the first counterexample to Question \ref{qBH}. Their precise result is as follows. Let $X:=G(3,6)$ be the Grassmannian parametrizing the $3$-dimensional subspaces of $\C^6$, and let $\mathcal{E}$ be the tautological rank $3$ vector bundle on~$X$. Then $c_2(\mathcal{E})\in \CH^2(X)$ is not a linear combination with integral coefficients of classes of smooth subvarieties of $X$.

  Since then, many other counterexamples to Question \ref{qBH} have been discovered, by combining topological obstructions as above with Hodge-theoretic arguments  (exploiting that the image of the cycle class map $\CH_d(X)\to H_{2d}(X,\Z)$ consists of Hodge classes).  
Debarre \cite{Debarre} constructed counterexamples on abelian varieties, and these examples were expanded by him and myself in \cite{BD}. In particular, we show in \cite[Corollary~1.3]{BD} that Question~\ref{qBH} may have a negative answer for codimension $2$ cycles on abelian varieties of dimension $6$ (as we explained in \S\ref{parpos}, this is the lowest possible dimension for the ambient variety~$X$). 

In addition, I have found counterexamples to Question \ref{qBH} with~$d=\frac{n}{2}$, for infinitely many values of $d$ (the smallest one being $d=6$).  For all these values of $d$, Theorem \ref{thKV} is therefore optimal. 
In the statement of this result, we let $\alpha(m)$ denote the number of ones in the binary expansion of the integer $m$.

\begin{thm}[{\cite[Theorem 0.3]{Ben}}]
\label{thsmooth}
Let $d\geq 0$ be such that $\alpha(d+1)\geq 3$.  Then there exists a smooth projective variety $X$ of dimension $2d$ over $\C$ such that $\CH_d(X)$ is not generated by classes of smooth subvarieties of $X$.
\end{thm}

The proof of Theorem \ref{thsmooth} is discussed in Section \ref{s4}.  For now, let us only explain in what way its principle relies on the Whitney heuristic discussed in~\S\ref{parWhitney}. 
Let~$X$ be a smooth projective variety of dimension $2d$ over $\C$. Fix $\beta\in\CH_d(X)$.  
As a first step towards proving Theorem \ref{thsmooth}, one might want to check that, if $\beta$ is the class of an integral subvariety $Z\subset X$, then $Z$ cannot be smooth.  To do so, imagine first that~$Z$ has been constructed by some kind of generic projection argument. In the range of dimensions we consider, the subvariety $Z$ would then have finitely many singular points at which it self\nobreakdash-intersects in~$X$.  To show that $Z$ cannot be smooth, compute this number of singularities using a double point formula,  and prove 
(only relying on the knowledge of~$\beta$ and not of~$Z$ itself) that this number must be nonzero. 
If $Z$ has more complicated singularities,  the double point formula still computes a virtual number of double points, allowing one to proceed in a similar~way.

For what values of $d$ (besides those found in Theorem \ref{thsmooth}) is Theorem \ref{thKV} optimal?  A full answer to this question is not known.
The first case that is open is $d=3$.

\begin{que}
Let $X$ be a smooth projective variety of dimension $6$ over $k$. Is the Chow group~$\CH_3(X)$ generated by classes of smooth subvarieties of $X$?
\end{que}

\subsubsection{Mildly singular cycles}
\label{mild}

It is also interesting to consider weakenings of Question \ref{qBH}, in which one allows subvarieties with controlled singularities.

Maggesi and Vezzosi put forward in \cite{MV} the case of subvarieties with locally complete intersection (lci) singularities.  (The article \cite{MV} works with Chow groups with rational coefficients,  but we defer until \S\ref{parrational} the consideration of rational Chow groups). It turns out that this variant of Question \ref{qBH} still has a negative answer.  In fact, in the counterexamples given by Theorem \ref{thsmooth}, the group $\CH_d(X)$ is not even generated by classes of lci subvarieties of $X$ (as we will prove in Section \ref{s4}, see Theorem \ref{thlci}).

In contrast, we do not know the answer to the next question.

\begin{que}
\label{qnormal}
Let $X$ be a smooth projective variety over $k$.  Are the Chow groups~$\CH_*(X)$ of $X$ generated by classes of normal subvarieties of $X$?
\end{que}

 In this direction, the arguments of Kleiman described in \S\ref{parrational} imply that $\CH_*(X)_{\Q}$ is generated by classes of normal subvarieties,  whereas the Koll\'ar--Voisin method shows that $\CH_*(X)$ is generated by classes of subvarieties whose normalization is smooth.

\subsection{Algebraic cycles with rational coefficients}
\label{parrational}

We conclude this introduction by discussing a fascinating
smoothing problem for algebraic cycles: the case of Chow groups with rational coefficients.

\begin{que}
\label{qQ}
Let $X$ be a smooth projective variety over $k$.  Are the rational Chow groups $\CH_*(X)_{\Q}$ of $X$ generated by classes of smooth subvarieties of $X$?
\end{que}

Let us first indicate that the analogue of Question \ref{qQ} in differential topology does have a positive answer. Namely, Thom has proved that, for any oriented compact~$\ci$ manifold $M$, the $\Q$-vector space $H_*(M,\Q)$ is generated by fundamental classes of oriented $\ci$ submanifolds of~$M$ (see \cite[Th\'eor\`eme II.29]{Thom}). This contrasts sharply with Thom's negative answer to Question~\ref{qThom}.  
In \cite[Theorem 0.4]{BV}, Voisin and I obtained a symplectic extension of Thom's result. More precisely,  we showed that if $M$ is a symplectic compact $\ci$ manifold, then the $\Q$-vector space $H_{2d}(M,\Q)$ is generated by fundamental classes of $d$-dimensional symplectic $\ci$ submanifolds of $M$. 
We believe that this result dashes any hope of finding topological obstructions to Question~\ref{qQ}.

Question \ref{qQ} was first investigated by Kleiman \cite{Kleiman},  using the techniques presented in \S\ref{lowc}.  Thanks to the Grothendieck--Riemann--Roch theorem without denominators (see  (\ref{GRR})), he reduces first to the case of Chern classes, and then to the case of cycles that are in the image of $f^*$ for some algebraic morphism $f:X\to G$ to a Grassmannian.  He further notes that $f$ may be chosen to be in generic position (with respect to any subvariety of $G$)  by composing it with a generic automorphism of~$G$. 

 On the one hand, this line of reasoning shows that it suffices to solve Question \ref{qQ} when $X$ is a Grassmannian.  This case is by no means easy: it is not even known if the second Chern class of the tautological bundle on $G(3,6)$ (the Hartshorne--Rees--Thomas counterexample to the original question of Borel and Haefliger, see \S\ref{count}) is a linear combination with rational coefficients of classes of smooth subvarieties of $G(3,6)$. 

On the other hand,  combining the above arguments with a study of the singularities of Schubert cells,  Kleiman answered positively Question~\ref{qQ} for cycles of dimension ${d<\frac{n}{2}+1}$ (see \cite[Theorem 5.8]{Kleiman}). In particular, the rationalized version of Theorem~\ref{thKV} has been known for a long time, with slightly better bounds on the dimension of the cycles.  It is striking that Kleiman's result has not been improved in the past fifty years.

In \cite[Theorem 0.3]{BV}, Voisin and I have put forward a curious connection between Question~\ref{qQ} and Hartshorne's famous conjecture on complete intersections in projective space (according to which any smooth subvariety of codimension $2$ in $\mathbb{P}^N_{\C}$ with~$N\gg0$  is a complete intersection; Hartshorne suggests the bound $N\geq 7$ but the assertion might hold for $N\geq 5$). Namely,  we prove that if Hartshorne's conjecture holds for some $N\geq 5$, then Question \ref{qQ} has a negative answer for codimension~$2$ cycles on the Grassmannian~$G(N,2N)$. We leave it to the readers to decide for themselves if they believe more in Hartshorne's conjecture or in a positive answer to Question \ref{qQ}.

\subsection{Notation and conventions}

 We work over a a field $k$ of characteristic $0$.  A variety over $k$ is a separated scheme of finite type over $k$. All the varieties that we consider are implicitly assumed to be equidimensional. By subvariety, we always mean closed subvariety.

\subsection{Acknowledgements}

  I thank J\'anos Koll\'ar and Claire Voisin for useful suggestions.

\section{Smoothing elements of $\CH^*(X)_{\KV}$}
\label{s1}

The aim of this section is to prove Theorem \ref{th1} by establishing the next proposition (see \cite[Proposition 2.1]{KV}).

\begin{pro}
\label{prosmoothing}
Let $f:Y\to X$ be a flat morphism of smooth projective varieties over $k$.  Let $\sL_1,\dots,\sL_c$ be very ample line bundles on $Y$.  For $1\leq i\leq c$, let $D_i\subset Y$ be a general member of the linear system $|\sL_i|$. Set $Z:=D_1\cap\dots\cap D_c$.
If $\dim(Y)-c<\frac{\dim(X)}{2}$, 
then $f|_Z:Z\to X$ is an embedding of smooth projective varieties over $k$.
\end{pro}

To deduce Theorem \ref{th1} from Proposition \ref{prosmoothing},  note that any class in~$\CH^1(Y)=\Pic(Y)$ is a difference of very ample line bundles. As a consequence,  under the hypotheses of Theorem \ref{th1}, any element of $\CH_d(X)_{\KV}$ can be represented by a linear combination with integral coefficients of smooth subvarieties of~$X$ constructed as in Proposition~\ref{prosmoothing}.

\vspace{.7em}

\begin{proof}
Let us first prove that $f|_Z$ is injective.  The set of pairs $(y_1,y_2)\in Y^2$ such that~$y_1\neq y_2$ and $f(y_1)=f(y_2)$ is an algebraic variety of dimension ${2\dim(Y)-\dim(X)}$. Once such a pair $(y_1,y_2)$ is fixed, the codimension in $|\sL_1|\times\dots\times|\sL_c|$ of the set of~$(D_1,\dots, D_c)$ such that $y_1,y_2\in Z$ is equal to $2c$. The codimension,  in this parameter space, of the locus where $f|_Z$ is not injective is therefore
$\geq 2c-(2\dim(Y)-\dim(X))>0$.

The proof that $f|_Z$ is immersive is similar, and based on the stratification
$$Y_r:=\{y\in Y\mid \rk(df_y)\leq r\}$$
of $Y$ by the rank of the differential of $f$.  Fix $r\geq 0$. The set of pairs $(y,v)$ with $y\in Y_r$ and~$v\in T_yY$ such that $df_y(v)=0$ is an algebraic variety of dimension 
$$\leq \dim(Y_r)+(\dim(Y)-r)\leq \dim\big(\overline{f(Y_r)}\big)+(\dim(Y)-\dim(X))+(\dim(Y)-r).$$

The rank at any smooth point of $Y_r$ of the differential of $f|_{Y_r}:Y_r\to \overline{f(Y_r)}$ is $\leq r$. 
By generic smoothness of $f|_{Y_r}$, we deduce that $\dim\big(\overline{f(Y_r)}\big)\leq r$. It follows that the set of pairs $(y,v)$ we were considering has dimension $\leq 2\dim(Y)-\dim(X)$. 

Once such a pair $(y,v)$ is fixed, the codimension in $|\sL_1|\times\dots\times|\sL_c|$ of the set of~$(D_1,\dots, D_c)$ such that $v\in T_yZ$ is equal to $2c$. 
We deduce that the codimension,  in this parameter space, of the locus where $f|_Z$ is not immersive at some point of $Y_r$ is~${\geq 2c-(2\dim(Y)-\dim(X))>0}$.
\end{proof}

\section{Functorial properties of sci-images}
\label{s2}

Sci-images were introduced in Definition \ref{defKV2}.  In this section, we show that this class of subvarieties is stable under taking images by various classes of morphisms. 
We follow \cite[Propositions 3.7, 3.9 and 3.11]{KV} and \cite[Lemmas 13, 15 and~16]{Kollar}.

We let $X$ denote a smooth projective variety of dimension $n$ over~$k$.  

\begin{pro}
\label{hyp}
Let $i:H\hookrightarrow X$ be the inclusion of a smooth hypersurface.  Let~$Z\subset H$ be an sci-image. Then $i(Z)\subset X$ is an sci-image.
\end{pro}

\begin{proof}
Let $\Gamma\subset H\times X$ be the graph of $i$.  Let $\nu: B\to H\times X$ be the blow-up of $\Gamma$.
The proof of the proposition relies on the commutative diagram
\begin{equation*}
\begin{aligned}
\xymatrix@C=1.5em@R=3ex{
&B\ar^(.56){\nu}[d]\ar_{q_H}[ldd]\ar^{q_X}[rdd]& \\
&H\times X\ar^{p_H}[ld]\ar_{p_X}[rd]& \\
H&&X\rlap{,}
}
\end{aligned}
\end{equation*}
in which $p_H$ and $p_X$ are the two projections. 
We claim that $q_H$ is smooth and that~$q_X$ is flat. To prove it, note that all the fibers of $q_H$ are smooth of pure dimension~$n$,  and that all the fibers of $q_X$ are of pure dimension $n-1$. 
More precisely, the fiber of~$q_H$ over~$x\in H$ is the blow-up of $X$ at $i(x)$.    
If $x\in H$, the fiber of~$q_X$ over~$i(x)$ has two irreducible components: one is the blow-up of~$H$ at $x$, and the other is the $(n-1)$-dimensional projective space $\nu^{-1}(x,i(x))$. The other fibers of~$q_X$ are isomorphic~to~$H$.  

Since $Z\subset H$ is an sci-image, there exist a flat morphism $f:Y\to H$ of smooth projective varieties over $k$ and a smooth complete intersection $V\subset Y$ such that ${f(V)=Z}$ and $f|_V:V\to Z$ is birational.  Let $f_B:Y_B\to B$ and $V_B\subset Y_B$ be constructed from~$f$ and $V$ by base change by $q_H:B\to H$.

Let $E\subset B$ be the exceptional divisor of $\nu$. Fix an ample line bundle~$\sL$ on~${H\times X}$. For~$l\gg 0$, the line bundle $\nu^*\sL^{\otimes l}(-E)$ on $B$ is very ample.  Pick general elements~$D_1,\dots, D_{n-1}$ in the linear system $|\nu^*\sL^{\otimes l}(-E)|$.
These choices imply that the morphism~$\hat{f}:=q_X\circ f_B:Y_B\to X$ is flat (as a composition of flat morphisms), and that the smooth complete intersection 
$$\widehat{V}:=V_B\cap f_B^{-1}(E)\cap f_B^{-1}(D_1)\cap\dots\cap f_B^{-1}(D_{n-1})\subset Y_B$$
is such that $\hat{f}(\widehat{V})=i(Z)$ and $\hat{f}|_{\widehat{V}}:\widehat{V}\to i(Z)$ is birational.  It follows that the subvariety~$i(Z)\subset X$ is an sci-image.
\end{proof}

Recall that a cbs in a smooth projective variety $X$ over $k$ is a subvariety of codimension $c$ of $X$ that is the zero locus of a section of a vector bundle of rank $c$ on~$X$.

\begin{pro}
\label{cbs}
Let $i:C\hookrightarrow X$ be the inclusion of a smooth cbs.  Let $Z\subset C$ be an sci-image. Then $i(Z)\subset X$ is an sci-image.
\end{pro}

\begin{proof}
Let $c$ be the codimension of $C$ in $X$. We argue by induction on $c$.  We have already dealt with the case~$c=1$ in Proposition \ref{hyp}, so we may assume that $c\geq 2$.

Write $C=\{s=0\}$, where $\sE$ is a vector bundle of rank $c$ on $X$ and  $s\in H^0(X,\sE)$.  
Let $\pi:\bP(\sE)\to X$ be the projective bundle parametrizing lines in $\sE$.  Consider the short exact sequence
$$0\to \sL\to\pi^*\sE\to\sF\to 0$$
of vector bundles on $\bP(\sE)$, where $\sL$ is the tautological line subbundle of $\pi^*\sE$. The section $\pi^*s$ induces a section $s'\in H^0(\bP(\sE),\sF)$.  Its zero locus $X':=\{s'=0\}\subset\bP(\sE)$ identifies with the blowup $\nu:X'\to X$ of $X$ along $C$ and hence is a smooth cbs in~$\bP(\sE)$.  
In turn, the zero locus of $s|_{C'}\in H^0(C',\sL|_{C'})$ is precisely $\pi^{-1}(C)$.  We summarize the situation in the commutative diagram
\begin{equation}
\label{diag}
\begin{aligned}
\xymatrix@C=1.5em@R=3ex{
\pi^{-1}(C)\ar^(.6){a}[r]\ar_{\pi|_{\pi^{-1}(C)}}[d]&X'\ar^(.4){b}[r]\ar_(.45){\nu}[rd]&\bP(\sE)\ar^{\pi}[d] \\
C\ar_{i}[rr]&&X \rlap{}
}
\end{aligned}
\end{equation}
in which the morphisms $a$ and $b$ are the inclusions.

Since $Z\subset C$ is an sci-image,  so is $\pi^{-1}(Z)\subset \pi^{-1}(C)$ (to see it, base change the flat morphism and the smooth complete intersection appearing in the definition of an sci\nobreakdash-images by the smooth morphism $\pi|_{\pi^{-1}(C)}$).  
By the induction hypothesis (applied first to $a$ and then to $b$), the subvariety $
b\circ a(\pi^{-1}(Z))=\pi^{-1}(i(Z))$ of $\bP(\sE)$ is an sci\nobreakdash-image. 
Consequently, there exist a flat morphism $f:Y\to \bP(\sE)$ of smooth projective varieties over $k$ and a smooth complete intersection $V\subset Y$ such that $f(V)=\pi^{-1}(i(Z))$ and~$f|_V:V\to \pi^{-1}(i(Z))$ is birational.

Now let $\sL$ be an ample line bundle on $X$. For $l\gg 0$, the line bundle $\pi^*\sL^{\otimes l}\otimes \sO_{\bP(\sE)}(1)$ on $\bP(\sE)$ is very ample.  Pick general elements $D_1,\dots, D_{c-1}$ in the associated  linear system $|\pi^*\sL^{\otimes l}\otimes \sO_{\bP(\sE)}(1)|$.
These choices imply that the morphism $\hat{f}:=\pi\circ f:Y\to X$ is flat,  and that the smooth complete intersection 
$$\widehat{V}:=V\cap f^{-1}(D_1)\cap\dots\cap f^{-1}(D_{c-1})\subset Y$$
is such that $\hat{f}(\widehat{V})=i(Z)$ and $\hat{f}|_{\widehat{V}}:\widehat{V}\to i(Z)$ is birational.  We conclude that the subvariety $i(Z)\subset X$ is an sci-image.
\end{proof}

We finally reach the goal of this section, which was stated as Proposition \ref{pro1} in the introduction.

\begin{pro}
\label{blowup}
Let $\nu:X'\to X$ be the blow-up of a smooth cbs.  Let $Z'\subset X'$ be an sci-image. If $\nu|_{Z'}:Z'\to\nu(Z')$ is birational, then $\nu(Z')\subset X$ is an sci-image.
\end{pro}

\begin{proof}
Denote by $i:C\hookrightarrow X$ the inclusion of the blown up cbs and keep the notation of (\ref{diag}).  Since $Z'\subset X'$ is an sci-image, Proposition \ref{cbs} implies that so is~$b(Z')\subset\bP(\sE)$.  
Consequently, there exist a flat morphism $f:Y\to \bP(\sE)$ of smooth projective varieties over $k$ and a smooth complete intersection $V\subset Y$ such that $f(V)=b(Z')$ and~${f|_{V}:V\to b(Z')}$ is birational.
Since $\pi$ is flat and since $\pi|_{b(Z')}:b(Z')\to \nu(Z')$ is birational,  the flat morphism $\pi\circ f :Y\to X$ and the smooth complete intersection~${V\subset Y}$ 
certify that the subvariety~$\pi\circ b(Z')=\nu(Z')$ of $X$ is an sci-image.
\end{proof}

\section{Turning smooth subvarieties into smooth complete intersections}
\label{s3}

The  goal of this section is to prove the next proposition (see \cite[Theorem 4.2]{KV}).

\begin{pro}
\label{pro3}
Let $X$ be a smooth projective variety of dimension $n$ over $k$. Let~$Z\subset X$ be a smooth subvariety of dimension $d<\frac{n}{4}$. Then there exists a composition $\pi:X'\to X$ of blow-ups of smooth cbs of dimension $<d$ such that the strict transform~$Z'\subset X'$ of $Z$ is a union of connected  components of a smooth cbs $C\subset X'$.
\end{pro}

Note that Proposition \ref{pro3} seems slightly weaker than what was promised in the introduction (see Proposition \ref{pro2}),  in that $C$ is only required to be a union of connected components of a smooth cbs, and not a smooth complete intersection.  This issue can be resolved by a very simple blow-up trick.

\vspace{.7em}

\begin{proof}[Proof of Proposition \ref{pro2}]
Let the morphism $\pi:X'\to X$ and the subvarieties $Z'$ and~$C$ of $X'$ be as in Proposition \ref{pro3}. Let~${\nu:X^+\to X'}$ be the blow-up of $C$ and let~$E\subset X^+$ be the union the exceptional divisors of $\nu$ that lie above $Z'$.
Fix an ample line bundle~$\sL$ on~$X'$. For~$l\gg 0$, the line bundle $\nu^*\sL^{\otimes l}(-E)$ on~$X^+$ is globally generated. Pick general elements $D_1,\dots, D_{c-1}$ in the linear system $|\nu^*\sL^{\otimes l}(-E)|$, where $c$ is the codimension of~$Z$ in $X$ (and hence of $Z'$ in $X'$).
The morphism $\pi^+:=\pi\circ\nu:X^+\to X$ and the smooth complete intersection
$$V:=E\cap D_1\cap\dots\cap D_{c-1}\subset X^+$$
then have the properties requested in the statement of Proposition \ref{pro2}.
\end{proof}

In \S\ref{simplified},  we describe a strategy to prove Proposition \ref{pro3}. This outline
will turn out to be insufficient, and we will explain in \S\ref{parmieux} and \S\ref{inductive} how this strategy must be modified to give rise to a rigorous proof.  The hypothesis that $Z$ has dimension $<\frac{n}{4}$ will only play a role in \S\ref{parmieux} and \S\ref{inductive}. The proofs of a few lemmas are postponed to \S\ref{parlemmas}.

\subsection{A simplified sketch of the proof of Proposition \ref{pro3}}
\label{simplified}

\begin{Step}[Induction on $d$]
We argue by induction on $d$. The case when~${d=0}$ is obvious: it is not even necessary to blow up $X$. From now on,  fix~$d\geq 1$.

During the whole proof we will repeatedly blow up $X$ along smooth cbs of dimension~${<d}$ and replace~$Z$ with its strict transform in these blow-ups.  To avoid creating singularities on $Z$, we always assume that the connected components of the blown-up cbs are either included in $Z$ or disjoint from $Z$.
The arguments given in this paragraph are not strong enough to ensure that.  This is one of the reasons why they will need to be modified as described in \S\ref{parmieux} and \S\ref{inductive}.
\end{Step}

\begin{Step}[Induction on $e$]
\label{step2}
Fix $d\leq e\leq n$.  We claim that, after maybe having repeatedly blown up $X$ along smooth cbs and replaced $Z$ by its strict transform,  there exists a smooth cbs $Y_e\subset X$ of dimension~$e$ that contains $Z$.  This claim is proved by decreasing induction on $e$.  The case $e=n$ is obvious: just take $Y_n=X$. The case~$e=d$ is exactly the conclusion of Proposition \ref{pro3} that we are trying to reach.
\end{Step}

\begin{Step}[A hyperplane section]
\label{step3}
To perform the induction step of the proof of the claim, we assume that $Z\subset Y_{e+1}\subset X$ has been constructed for some $d\leq e< n$.  
Choose a sufficiently ample line bundle $\sL$ on $X$.  Let $s\in H^0(X,\sL)$ be general section vanishing on $Z$ and set~$H:=\{s=0\}$. Then $H\subset X$ is a hypersurface containing $Z$.  Define~$Y_{e}:=Y_{e+1}\cap H$. 

The proof is not yet complete because the subvariety $Y_{e}$ may have singularities.
Local computations (more precisely: Lemma~\ref{quadsing} applied to the variety~$Y_{e+1}$) allow one to describe them precisely.
Let~$\Sigma\subset Y_{e}$ denote the singular locus of $Y_{e}$.  Note that~$\Sigma\subset Z$ by the Bertini theorem.
 If~$e\geq 2d$, then $\Sigma=\varnothing$ and the proof is finished. 
Otherwise, the variety~$\Sigma$ is smooth of dimension $2d-e-1$ (in particular, of dimension~$<d$)
and~$Y_{e}$ has ordinary quadratic singularities along $\Sigma$. 
\end{Step}

\begin{Step}[Making use of the induction hypothesis]
\label{step4}
This description of the singularities of $Y_e$ shows that, to render $Y_e$ smooth, it would suffice to blow up $\Sigma$. However, we are only allowed to blow up smooth cbs and not arbitrary smooth subvarieties. 

To overcome this obstacle, we apply the induction hypothesis on $d$.  After repeatedly blowing up smooth cbs of dimension $<\dim(\Sigma)$ (whose components we assume, for the purpose of this simplified sketch of proof,  to be either included in $\Sigma$ or disjoint from~$Y_{e+1}$)
and replacing all the varieties in sight with their strict transforms in these blow-ups, we may assume that $\Sigma\subset X$ is a union of connected components of a smooth~cbs.

It is easy to check that, after this sequence of blow-ups, the equality $Y_{e}=Y_{e+1}\cap H$ still holds.
One further verifies that the subvariety $Y_{e+1}\subset X$ is still a cbs (see Lemma~\ref{encorecbs} (i)) and hence that so is $Y_{e}=Y_{e+1}\cap H$, and that $Y_e$ still has ordinary quadratic singularities along~$\Sigma$ (see Lemma \ref{encorequad}).
\end{Step}

\begin{Step}[Blowing up the singular locus]
\label{step5}
Let $C\subset X$ be the smooth cbs of which~$\Sigma$ is a union of connected components.  Again, we assume for the purpose of this simplified sketch of proof that $C$ only intersects $Y_e$ along $\Sigma$.
Blowing up $C$ in $X$ and replacing~$Y_e$ with its strict transform in this blow-up then completes the proof. Indeed,  our choice of blow-up center implies that the subvariety~$Y_e\subset X$ is now a smooth cbs (see Lemma~\ref{encorecbs}~(ii)).
\end{Step}

\subsection{Controlling the blown-up loci}
\label{parmieux}

At several points in the sketch of proof described in \S\ref{simplified},  we took for granted that our blow-up centers
were in good position with respect to various other subvarieties~$V_i\subset X$ of
importance in the proof. It is for this reason, and for this reason alone, that this sketch is incomplete.

Following \cite[Definition 4.3]{KV}, we say that a subvariety $C\subset X$ has \textit{full intersection} with another subvariety $V\subset X$ if all the connected components of $C$ are either included in $V$ or disjoint from $V$.  The key to fixing the proof will be to ensure that our blow-up centers have full intersection with the relevant subvarieties $V_i\subset X$.
One could hope that if the choices made in \S\ref{simplified} are generic, then the blown up loci satisfy the required full intersection conditions.
Unfortunately, as noted in~\mbox{\cite[Remark 4.7]{KV}},  this is not~clear.

The solution of Koll\'ar and Voisin to this difficulty is radical. Instead of making generic choices  and hoping for the best,  they make particular choices.  To be precise, they constrain the blown up loci inside auxiliary smooth subvarieties of $X$ that are as disjoint as possible from the subvarieties $V_i$ of $X$ that are relevant to the proof.  This forces the blow-up centers to have full intersection with the $V_i$.  

This strategy can only be made to work if the $V_i$ and the auxiliary varieties all have dimension~$<\frac{n}{2}$ (they will then be as disjoint as possible from each other for purely dimensional reasons).  Since the auxiliary varieties will be constructed as general complete intersections containing $\Sigma$ (which is smooth of dimension~$<d$),  their smoothness is ensured if $d<\frac{n}{4}$ (by Lemma \ref{quadsing}). This 
explains why this hypothesis appears the statement of Proposition \ref{pro3}.

\subsection{The inductive statements}
\label{inductive}

In practice,  the resulting enhancement of Proposition \ref{pro3} is stated as the combination of two twin propositions (see \cite[Properties 4.6 and 4.8]{KV}), which are proved simultaneously by induction on $d$.   Their statements is complicated by the necessity of keeping track of all the auxiliary varieties that have been introduced at various stages of the induction process.

As before, we let $X$ be a smooth projective variety of dimension $n$ over $k$.

\begin{pro}
\label{pro4}
Let $Z\subset Y\subset X$ and $Z\subset W_j\subset X$ be subvarieties.  Suppose that~$Z$ is smooth of dimension $d<\frac{n}{4}$,  that the $W_j$ are smooth of dimension $<\frac{n}{2}$ and that~$Y$ is a smooth cbs of dimension $<\frac{n}{2}$.  

Then there exist a composition $\pi:X'\to X$ of blow-ups of smooth cbs of dimension~${<d}$ that have full intersection with the strict transforms of $Z$, 
of $Y$ and of the~$W_j$, and a smooth cbs $C\subset X'$,  such that $Z'\subset C\subset Y'$ and $Z'$ is a union of connected components of $C$ (where $Z'$ and $Y'$ denote the strict transforms of $Z$ and $Y$ in $X'$).
\end{pro}

\begin{pro}
\label{pro5}
Let $Z\subset Y\subset X$ and $Z\subset W_j\subset X$ be subvarieties.  Suppose that~$Z$ is smooth of dimension $d<\frac{n}{4}$,  that the $W_j$ are smooth of dimension $<\frac{n}{2}$ and that~$Y$ is a smooth cbs of dimension $<\frac{n}{2}$.  Suppose moreover that $Y$ has dimension $>d$.

Then there exist a composition $\pi:X'\to X$ of blow-ups of smooth cbs of dimension~${<d}$ that have full intersection with the strict transforms of $Z$, 
of $Y$ and of the $W_j$, and a smooth cbs $C\subset X'$, such that $Z'\subset C\subset Y'$ and $\dim(C)=\dim(Y')-1$ (where~$Z'$ and $Y'$ denote the strict transforms of $Z$ and $Y$ in $X'$).
\end{pro}

Proposition \ref{pro3} follows from Proposition \ref{pro4} applied with no $W_j$ and with $Y$ a general complete intersection containing $Z$ (which may be chosen smooth thanks to Lemma \ref{quadsing}). 

Proposition \ref{pro4} (for a fixed $d$) is an immediate consequence of Proposition \ref{pro5} (for the same $d$). 
 To see it, apply Proposition \ref{pro5} and replace $X$ with $X'$, the $W_j$ and~$Z$ with their strict transforms in $X'$, and $Y$ with $C$. In this way the dimension of~$Y$ drops. Repeat this process until $\dim(Y)=d$ and the conclusion of Proposition \ref{pro4} is reached.  

We finally explain how Proposition \ref{pro5} (for a fixed $d$) can be deduced from Proposition~\ref{pro4} (for lower values of $d$). The argument is (a rigorous variant of) the induction step on the parameter $e$ that was presented in Steps \ref{step3}, \ref{step4} and \ref{step5} of \S\ref{simplified}.  Let $H\subset X$ be a general hypersurface containing $Z$ and set $C:=Y\cap H$.  The singular locus $\Sigma$ of~$C$ is a smooth subvariety of dimension $<d$ of $Z$ (see Lemma \ref{quadsing}). 
Our plan is to turn $\Sigma$ into a union of components of a smooth cbs and then to blow up this smooth~cbs.

As we explained in \S\ref{parmieux},  the difficulty is to control the location in $X$ of the other components of this smooth cbs. To this effect, we introduce an auxiliary variety: a general complete intersection $\overline{Y}\subset X$ containing $\Sigma$, chosen smooth of dimension~$<\frac{n}{2}$ (this is possible by Lemma \ref{quadsing}). 
A simple dimension count shows that $\overline{Y}\cap Y=\Sigma$ and~$\overline{Y}\cap W_j=\Sigma$ scheme-theoretically.

We now apply Proposition \ref{pro4} to $\Sigma\subset \overline{Y}\subset X$, with $Z$
and $Y$ added to the collection of the $W_j$.  In this manner,  after repeatedly blowing up smooth cbs (with full intersection with the relevant subvarieties) and replacing all the varieties in sight by their strict transforms, we can arrange that $\Sigma$ is a union of connected components of a smooth cbs~$C_{\Sigma}\subset X$ which is included in~$\overline{Y}$.  

The equality $C=Y\cap H$ still holds after this sequence of blow-ups.  The subvariety~$Y\subset X$ is still a cbs (by Lemma~\ref{encorecbs} (i)), and hence so is $C=Y\cap H\subset X$. Moreover,  the variety $C$ still has ordinary quadratic singularities along~$\Sigma$ (by Lemma \ref{encorequad}).

Local computations show that the scheme-theoretic equalities $\overline{Y}\cap Y=\Sigma$ and ${\overline{Y}\cap W_j=\Sigma}$ are preserved during the blow-up process.
 The cbs $C_{\Sigma}$ can therefore only intersect $Z$,  $Y$ and the $W_j$ along $\Sigma$ and hence has full intersection with them. It is therefore legitimate to blow it up.  The strict transform of $C$ in this blow-up is a smooth cbs (by Lemma \ref{encorecbs} (ii)) which has the properties required in the statement of Proposition \ref{pro5}.

\subsection{Useful lemmas}
\label{parlemmas}

We gather here a few lemmas that were used in the proof of Proposition \ref{pro3}. We continue to let $X$ denote a smooth projective variety of dimension $n$ over $k$.

\subsubsection{Ordinary quadratic singularities}

The next two lemmas are standard, and we omit their proofs. The first is a Bertini-type theorem appearing in \cite[Theorem 2.1, Corollary 2.5]{DH} (see also \cite[Lemma 3.13]{KV}). The second is a simple local computation for which we refer to \cite[Lemma 4.4]{KV}.

\begin{lem}
\label{quadsing}
Let~$Z\subset X$ be a smooth subvariety of dimension $d<n$. Let $\sL$ be a sufficiently ample line bundle on $X$. Let $H\in|\sL|$ be a general hypersurface containing~$Z$.

If $d<\frac{n}{2}$, then $H$ is smooth.  If $d\geq \frac{n}{2}$, then $H$ is singular along a smooth subvariety of dimension $2d-n$ of $Z$ and has ordinary quadratic singularities along this subvariety.
\end{lem}

\begin{lem}
\label{encorequad}
Let $Y\subset X$ be a subvariety with smooth singular locus $\Sigma\subset Y$, which has ordinary quadratic singular along $\Sigma$. Let $Z\subset \Sigma$ be a smooth subvariety of dimension~${<\dim(\Sigma)}$. Let $\pi:X'\to X$ be the blow-up of $\Sigma$. Then the strict transform~$Y'\subset X'$ of~$Y$ has ordinary quadratic singularities along the strict transform $\Sigma'\subset X'$ of $\Sigma$.
\end{lem}

\subsubsection{Blowups of complete bundle sections}

The following lemma (for which see \mbox{\cite[Lemmas  4.4 and 4.5]{KV}}) is crucial to the mechanism of proof of Proposition~\ref{pro3}. In particular,  Lemma \ref{encorecbs} (ii) is the reason why it is necessary to blow up smooth cbs and not only smooth complete intersections in Proposition \ref{pro3} (and hence in Proposition~\ref{pro2}).

\begin{lem}
\label{encorecbs}
 Let $C\subset X$ be a cbs.  Let $\pi:X'\to X$ be the blow-up of a smooth subvariety $Z\subset C$ of dimension $<\dim(C)$. 
Let $C'\subset X'$ be the strict transform of $C$.

(i) If $C$ is smooth, then $C'\subset X'$ is a smooth cbs.

(ii) If $Z$ is the singular locus of $C$ and $C$ has ordinary quadratic singularities along~$Z$, then $C'\subset X'$ is a smooth cbs.
\end{lem}

\begin{proof}
In both cases, it is clear that $C'$ is smooth and we must check that it is a cbs.

Let $c$ be the codimension of $C$ in $X$.  Write $C=\{s=0\}$, where ${s\in H^0(X,\sE)}$ for some vector bundle $\sE$ of rank~$c$ on~$X$.  Let $E$ be the exceptional divisor of $\pi$.
As the section $\pi^*s\in H^0(X',\pi^*\sE)$ vanishes on $E$, it gives rise to a section~$s'\in H^0(X',\pi^*\sE(-E))$. 

In case (i), a local computation shows that $C'=\{s'=0\}$, which concludes the proof.

In case (ii), the inclusion $C'\subset\{s'=0\}$ is strict and one must modify $\pi^*\sE(-E)$ to reach the desired conclusion.
Differentiating the section $s$ along $Z$ gives rise to a morphism~${ds:N_{Z/X}\to \sE|_Z}$ of vector bundles of rank $c$ on $Z$.  
Since $C$ has hypersurface singularities along $Z$, the kernel of the linear map $(ds)_x$ has dimension $1$ for all $x\in Z$. We deduce that cokernel of $(ds)_x$ also has dimension $1$ for all $x\in Z$,  and hence that the cokernel of $ds$ is a line bundle $\sL$ on $Z$. Let $\sF$ be the kernel of the composition
$$\pi^*\sE(-E)\to (\pi^*\sE(-E))|_E=(\pi|_E)^*\sE|_Z (-E)\to (\pi|_E)^*\sL (-E)$$
of the restriction map and the quotient map.  Computations based on the shape of local equations for subvarieties with ordinary quadratic singularities then show that $\sF$ is a vector bundle on $X'$, that the image of $s'\in H^0(X',\pi^*\sE(-E))$ in $H^0(X', (\pi|_E)^*\sL (-E))$ vanishes (so $s'$ lifts to a section $t\in H^0(X',\sF)$), and that $C'=\{t=0\}$. 
\end{proof}

\section{Half-dimensional cycles that are not smoothable}
\label{s4}

In this last section, we illustrate the sharpness of the Koll\'ar--Voisin technique by explaining the ideas of the proof of the following theorem. Recall that $\alpha(m)$ denotes the number of ones in the binary expansion of the integer $m$.

\begin{thm}
\label{thlci}
Let $d\geq 0$ be such that $\alpha(d+1)\geq 3$.  Then there exists a smooth projective variety $X$ of dimension $2d$ over $\C$ such that $\CH_d(X)$ is not generated by classes of lci subvarieties
of pure dimension $d$ in $X$.
\end{thm}

The varieties used in the proof of Theorem \ref{thlci} are provided by the next proposition (see {\cite[Proposition 4.13]{Ben}}).

\begin{pro}
\label{propcons}
For all $d\geq 1$, there exist a connected smooth projective variety~$X$ of dimension $2d$ over $\C$  and a class $\beta\in \CH_d(X)$ such that:
\begin{enumerate}[(i)]
\item If $\gamma,\gamma'\in H^{2d}(X(\C),\Z)$ are Hodge, then $\degr(\gamma\cdot\gamma')$ is even.
\item One has $\degr(\beta^2)\equiv 2\pmod 4$.
\item All the higher Chern classes of $X$ are torsion, that is $c(T_X)=1$ in $\CH^*(X)\otimes_{\Z}\Q$.
\end{enumerate}
\end{pro}

\begin{proof}
Let $(A,\lambda)$ be a very general principally polarized abelian variety of dimension~$d$ over $\C$ and let $\tau\in A(\C)$ be a $2$-torsion point.  Let $\Z/4$ act on $A\times A$ by means of the order $4$ automorphism $\phi:(x,y)\mapsto (y+\tau, x)$.  Set $X:=(A\times A)/(\Z/4)$ and let~$p:A\times A\to X$ be the quotient morphism.  In addition choose $\beta\in\CH_d(X)$ to be the class of $p(A\times\{0\})$ in $\CH_d(X)$.

Since $\Z/4$ acts freely on $A\times A$, the quotient morphism $p$ is \'etale.  As the tangent bundle of $A$ is trivial, one computes that $p^*c(T_X)=c(p^*T_X)=c(T_{A\times A})=1$. Apply~$p_*$ to that identity to show (iii).

The class $p^*\beta$ is represented by the $\Z/4$-orbit of~$A\times\{0\}$. A direct computation therefore shows that~$\deg(p^*\beta^2)=8.$ Since $\deg(p)=4$, one gets $\deg(\beta^2)=2$, proving~(ii).

It remains to prove (i).  The detailed argument being quite lengthy, we only explain its principle and refer to  \cite[\S 4.2]{Ben} for more details.   Since the principally polarized abelian variety $(A,\lambda)$ is very general,  a Mumford--Tate group argument (see~\mbox{\cite[Lemma~4.6]{Ben}}) shows that the algebra $\Hdg^{2*}(A(\C)\times A(\C),\Q)$ of $\Q$-Hodge classes on $A\times A$ is generated by the principal polarizations of both factors and by the class of the Poincar\'e bundle (where we identify the second factor to the dual of $A$ by means of~$\lambda$). Using this piece of information, one can compute a~$\Z$-basis of the group $\Xi:=\Hdg^{2d}(A(\C)\times A(\C),\Z)^{\Z/2}$ of~$\Z$-Hodge classes of degree $2d$ on $A\times A$ that are fixed by the involution exchanging the two factors (see \cite[Lemma 4.7 (ii)]{Ben}).  A direct computation on these basis elements shows that~$\deg(\delta\cdot\delta')\equiv 0\pmod 2$ for all $\delta,\delta'\in \Xi$ (see \cite[Lemma 4.9]{Ben}).  Now, if~$\gamma,\gamma'$ are elements of $\Hdg^{2d}(X(\C),\Z)$, one verifies that the elements $p^*\gamma$ and $p^*\gamma'$ of $\Xi$ are divisible by $2$ in $\Xi$ (see \cite[Lemma 4.7 (iii)]{Ben}). Applying the above congruence to $\delta:=\frac{p^*\gamma}{2}$ and $\delta':=\frac{p^*\gamma'}{2}$ yields $\deg(p^*\gamma\cdot p^*\gamma')\equiv 0\pmod 8$.
Using that $\deg(p)=4$, we get $\deg(\gamma\cdot\gamma')\equiv 0\pmod 2$.
\end{proof}

\begin{proof}[Proof of Theorem \ref{thlci}]
Let $X$ and $\beta$ be as in Proposition \ref{propcons}.  We assume that 
$$\beta=\sum_i n_i[Z_i]\textrm{\, in\, }\CH_d(X),$$
where $n_i\in\Z$ and the $Z_i$ are lci subvarieties of $X$ of dimension $d$, and we seek a contradiction.  Suppose first that the $Z_i$ are smooth. Then
\begin{equation}
\label{squared}
\deg(\beta^2)=\sum_i n_i^2\deg([Z_i]^2)+2\sum_{i<j}n_in_j\deg([Z_i]\cdot[Z_j]).
\end{equation}

The double point formula \cite[Corollary~6.3]{Fulton} shows that 
$\deg([Z_i]^2)=\deg(c_d(N_{Z_i/X}))$.
As $c(TX)=1$ in $\CH^*(X)\otimes_{\Z}\Q$ by Proposition \ref{propcons} (iii), we see that 
$$c(N_{Z_i/X})=c(T_{Z_i})^{-1}=s(T_{Z_i})$$ 
in $\CH^*(X)\otimes_{\Z}\Q$, where $s(T_{Z_i})$ denotes the total Segre class of $T_{Z_i}$.
A topological computation due to Rees and Thomas (see \mbox{\cite[Theorem 3]{RT}}) shows that, under the hypothesis that ${\alpha(d+1)\geq 3}$, the number $\deg(s_d(T_{Z_i}))$ is divisible by $4$. This implies that the term $\sum_i n_i^2\deg([Z_i]^2)=\sum_i n_i^2\deg(c_d(N_{Z_i/X}))=\sum_i n_i^2\deg(s_d(T_{Z_i}))$ in (\ref{squared}) is divisible by~$4$.

On the other hand, the numbers $\deg([Z_i]\cdot[Z_j])$ are even by Proposition~\ref{propcons}~(i) because the image of the cycle class map consists of Hodge classes. So the term $2\sum_{i<j}n_in_j\deg([Z_i]\cdot[Z_j])$ in (\ref{squared}) is also divisible by $4$.  We deduce that the left-hand side~$\deg(\beta^2)$ of~(\ref{squared}) must be divisible by $4$, which contradicts Proposition~\ref{propcons}~(ii).

The above proof can be adapted to the case where the $Z_i$ are only assumed to be lci. We now explain the required modifications. We let 
$N_{Z_i/X}$ and~$T_{Z_i}$ denote the normal bundle of the regular imbedding ${Z_i\hookrightarrow X}$ (see \mbox{\cite[B.7.1]{Fulton}}) and the virtual tangent bundle of $Z_i$ (see \mbox{\cite[\S 7.4.2]{LM}}) respectively. 
Fulton's double point formula \mbox{\cite[Corollary~6.3]{Fulton}} is stated in the required generality.  To justify that $\deg
(s_d(T_{Z_i}))$ is divisible by $4$ when~$Z_i$ is possibly not smooth, we make use of the algebraic cobordism group $\Omega_d$ defined in~\cite{LM}.  Every lci projective variety~$Z$ of dimension $d$ over $\C$ has a class $[Z]\in\Omega_d$ (see \mbox{\cite[\S 7.4.4]{LM}}).
By \mbox{\cite[Proposition 7.4.5]{LM}} (applied to the polynomial~$x_d$), there exists a group morphism~$s_d:\Omega_d\to \Z$ such that $s_d([Z])=\deg(s_d(T_Z))$ for any $Z$ as above. 
As the classes~$[Z]$ for $Z$ smooth generate $\Omega_d$ (see  \cite[Lemma~2.5.15]{LM}) and as~$s_d([Z])$ is divisible by $4$ when $Z$ is smooth by the Rees--Thomas theorem \mbox{\cite[Theorem 3]{RT}} (under the hypothesis that ${\alpha(d+1)\geq 3}$), we deduce, as desired, that~$s_d([Z])$ is also divisible by $4$ when $Z$ is only lci.
\end{proof}

\end{document}